\documentclass[12pt,a4paper,reqno]{amsart}
\usepackage{amsmath}
\usepackage{amsfonts}
\usepackage{amssymb,amsthm}
\numberwithin{equation}{section}

\def\Z{{\mathbb Z}}
\def\E{{\mathbb E}}

\def\R{{\mathbb R}}
\def\N{{\mathbb N}}
\def\Q{{\mathbb Q}}

\def\pmod #1{\ ({\rm{mod}}\ #1)}
\def\bx {{\mathbf x}}
\def\bB {{\mathbf B}}
\def\cP{{\mathcal P}}

\def\cA{{\mathcal A}}
\def\cR{{\mathcal R}}
\def\cS{{\mathcal S}}

\def\cJ{{\mathcal Z}}
\def\cI{{\mathcal I}}
\def\cJ{{\mathcal J}}

\theoremstyle{plain}
\newtheorem{theorem}{Theorem}
\newtheorem{lemma}{Lemma}

\newtheorem{proposition}{Proposition}

\theoremstyle{definition}
\newtheorem*{acknowledgment}{Acknowledgments}
\theoremstyle{remark}

\begin{document}

\title
{The Green-Tao theorem for primes of the form $x^2+y^2+1$}

\author{Yu-Chen Sun}
\address {Medical School, Nanjing
University, Nanjing 210093, People's Republic of China}
\email{b111230069@smail.nju.edu.cn}
\author{Hao Pan}
\address {Department of Mathematics, Nanjing
University, Nanjing 210093, People's Republic of China}
\email{haopan79@zoho.com}

 \begin{abstract}
 We prove that the primes of the form $x^2+y^2+1$ contain arbitrarily long non-trivial arithmetic progressions.
\end{abstract}
\maketitle

\section{Introduction}
\setcounter{lemma}{0}
\setcounter{theorem}{0}
\setcounter{corollary}{0}
\setcounter{remark}{0}
\setcounter{equation}{0}
\setcounter{conjecture}{0}

Let $\cP$ denote the set of all primes. The celebrated Green-Tao theorem \cite{GT08} asserts that $\cP$ contains arbitrarily long non-trivial arithmetic progressions. That is, for any $k\geq 3$, there exists positive integers $a$ and $d$ such that $a,a+d,\ldots,a+(k-1)d$ are all primes. In fact, they obtained a stronger result. For a subset $A\subseteq\cP$, define the relative upper density of $\cA$ by
$$
\overline{d}_\cP(\cA):=\limsup_{X\to\infty}\frac{|\cA\cap[1,X]|}{|\cP\cap[1,X]|}.
$$
Green and Tao proved that for any subset $\cA$ of primes with $\overline{d}_\cP(\cA)>0$, $A$ contains arbitrarily long non-trivial arithmetic progressions. There are three key ingredients in Green and Tao's proof: the Szemer\'edi theorem, a transference principle and a pseudorandom measure for primes.

Nowadays, the Green-Tao theorem has been generalized in different directions \cite{GT10, Le11, LW14, Tao06, TZ08,TZ015}. For example, a prime $p$ is called {\it Chen prime} provided that $p+2$ has at most two prime factors. The classical Chen theorem says that there exist infinitely many Chen primes. Green and Tao \cite{GT08} also claimed that using the similar discussions, one can prove that the Chen primes contains arbitrarily long non-trivial arithmetic progressions. And they gave a Fourier proof of the existence of infinitely many non-trivial three-term arithmetic progressions in the  Chen primes. Subsequently, the detailed proof of the extension of Green-Tao theorem to the Chen primes was given by Zhou in \cite{Zhou09}.

On the other hand, let us consider those primes which can be represented as the sum of two squares plus 1. Let $\cP_2$ denote the set of all such primes, i.e.,
$$
\cP_2:=\{p\text{ prime}:\, p=x^2+y^2+1,\ x,y\in\N\}.
$$
Linnik proved that $\cP_2$ contains infinitely many primes. In \cite{Iwaniec72}, Iwaniec proved that for any sufficiently large $X$,
$$
\big|\cP_2\cap[1,X]\big|\geq\frac{CX}{(\log X)^{\frac32}}
$$
for some constant $C>0$. For more about the primes in $\cP_2$, the reader may refer to \cite{Matomaki07,Matomaki08,Motohashi70,Wu98}. Recently, Ter\"av\"ainen \cite{Teravainen} proved that 
there exist infinitely many non-trivial three-term arithmetic progressions in any subset of $\cP_2$ with a positive relatively density, which extends the Roth-type theorem in the primes of Green \cite{Green05}.
It is natural to ask whether the Green-Tao theorem also can be extended to the primes in $\cP_2$. In this  paper, we shall give such an extension.
\begin{theorem}\label{main}
Suppose that $\cA$ is a subset of $\cP_2$ with the relatively density
$$
\overline{d}_{\cP_2}(\cA):=\limsup_{X\to\infty}\frac{|\cA\cap[1,X]|}{|\cP_2\cap[1,X]|}>0.
$$
Then $\cA$ contains arbitrarily long non-trivial arithmetic progressions.
\end{theorem}
The key to our proof is to construct a pseudorandom measure for those primes in $\cP_2$. In the next section, we shall first give the construction of such a pseudorandom measure. Then the proof of Theorem \ref{main} can be reduced to a Goldston-Y{\i}ld{\i}r{\i}m-type \cite{GY03} estimation, which will be proved in the third section.

We introduce several notions which will be used later. Suppose that $S$ is a finite subset and $f(x)$ is a function over $S$. Write
$$
E\big(f(x)|x\in S\big):=\frac{1}{|S|}\sum_{x\in S}f(x).
$$
For an assertion $P$, set ${\mathbf 1}_P=1$ or $0$ according to whether $P$ holds or not. Also, let $\phi$ denote the Euler totient function and let $\mu$ denote the M\"obius function.
\section{The pseudorandom measure}
\setcounter{lemma}{0}
\setcounter{theorem}{0}
\setcounter{corollary}{0}
\setcounter{remark}{0}
\setcounter{equation}{0}
\setcounter{conjecture}{0}

First, let us introduce the definition of the linear forms condition and the correlation condition. Suppose that $v\in\Z$ and $L_1,\ldots,L_{k}\in\Q$ are rational number whose numerators and denominators are all bounded. We call
$$
\psi(\bx):=L_1x_1+\cdots+L_kx_k+v
$$
a {\it linear form}, where $\bx=(x_1,\ldots,x_k)$. Suppose that $N$ is a sufficiently large prime. Let $\Z_N:=\Z/N\Z$ be the cyclic group of order $N$. Then $\psi$ also can be viewed as a linear form over $\Z_N$. Let $\nu:\,\Z_N\to\R$ be a non-negative function. Suppose that $h_0,k_0,m_0$ are positive integers. Suppose that 
$$
\E\big(\nu(\psi_1(\bx))\cdots\nu(\psi_h(\bx))\big|\bx\in\Z_N^k\big)=1+o_{h_0,k_0,m_0}(1)
$$
for any $1\leq h\leq h_0$, $1\leq k\leq k_0$ and the linear forms $\psi_i(\bx)=L_{i,1}x_1+\cdots+L_{i,k}x_k+v_i$ with the numerators and denominators of those $L_{i,j}$
all lie in $[-m_0,m_0]$. Then we say that $\nu$ obeys the {\it $(h_0,k_0,m_0)$-linear forms condition}. Similarly, suppose that for any $1\leq k\leq k_0$, there exists a non-negative weight function $\tau_k:\,\Z_N\to\R$ such that $\E(\tau_k(x)^s|x\in\Z_N)=O_{k,s}(1)$, and
$$
\E\big(\nu(x+v_1)\cdots\nu(x+v_k)\big| x\in\Z_N\big)\leq\sum_{1\leq i<j\leq k}\tau_k(v_i-v_j)
$$
for any $v_1,\ldots,v_k\in\Z_N$. Then say that $\nu$ obeys the {\it $k_0$-correlation condition}. 
$\nu(x)$ is call a {\it $m$-pseudorandom measure}, if $\nu$ both satisfies $(2^{m-1}m,3m-4,m)$-linear forms condition and $2^{m-1}$-correlation condition.
Green and Tao  
\cite[Theorem 3.5]{GT08} proved the following relatively Szemer\'edi theorem.
\begin{lemma} 
Suppose that $\delta>0$ and $m\geq 3$. Let $f(x)$ be a function over $\Z_N$ such that 
$\E\big(f(x)\big|x\in\Z_N\big)\geq\delta$, and 
$0\leq f(x)\leq \nu(x)$
for each $x\in\Z_N$, where $\nu$ is a $m$-pseudorandom measure over $\Z_N$. Then as $N\to\infty$,
\begin{equation}\label{fkaa}
\E\big(f(x)f(x+y)\cdots f\big(x+(m-1)y\big|x,y\in\Z_N\big)\geq c(m,\delta)+o_{k,\delta}(1),
\end{equation}
where $c(m,\delta)$ is a constant only depending on $m$ and $\delta$.
\end{lemma}
In \cite{CFZ15}, Conlon, Fox and Zhao improved Green-Tao's transference principle and showed that the requirement concerning the correlation condition is factly unnecessary. Suppose that as $N\to\infty$,
$$
\E\bigg(\prod_{j=1}^k\prod_{\omega\in\{0,1\}^{\{1,\ldots,k\}\setminus\{j\}}}\nu\bigg(\sum_{i=1}^k(i-j)x_i^{(\omega_i)}\bigg)^{\varrho_{j,\omega}}\bigg|x_1^{(0)},x_1^{(1)},\ldots,x_k^{(0)},x_k^{(1)}\in\Z_N\bigg)=1+o_k(1)
$$
for any choice of $\varrho_{j,\omega}\in\{0,1\}$. Then we say $\nu$ obeys the {\it $k$-linear forms condition}.
Conlon, Fox and Zhao proved that
\begin{lemma}\label{CFZThm}
Suppose that $\delta>0$ and $m\geq 3$. Let $f(x)$ be a non-negative function over $\Z_N$ such that 
$\E\big(f(x)\big|x\in\Z_N\big)\geq\delta$
and  $f(x)\leq \nu(x)$
for some function $\nu$ obeys the $m$-linear forms condition. Then (\ref{fkaa}) is also valid.
\end{lemma}
Clearly the $(2^{m-1}m,2m,m)$-linear forms condition is stronger than the $m$-linear forms condition. So for every $m\geq 3$, we need to construct a pseudorandom measure obeying the $(2^{m-1}m,2m,m)$-linear forms condition for those primes in $\cP_2$.

For any positive integer $q$, let
$$
\cR_q=\{p\text{ prime}:\,p=q^2x^2+q^2y^2+1,\text{ where }x,y\in\N\text{ and }(x,y)=1\}.
$$
In \cite{Iwaniec72}, Iwaniec proved that 
\begin{equation}\label{Iwaniec}
\frac{c_1}{\phi(q^2)}\cdot\frac{X}{(\log X)^{\frac32}}\leq |\cR_q\cap[1,X]|\leq \frac{c_2}{\phi(q^2)}\cdot\frac{X}{(\log X)^{\frac32}}
\end{equation}
for any sufficiently large $X$, where $c_1,c_2>0$ are constants. Of course, by following Iwaniec's discussions, we may easily obtain that
\begin{equation}\label{Iwaniec2}
\frac{C_1}{\phi(q^2)}\cdot\frac{X}{(\log X)^{\frac32}}\leq |\cR_q\cap[X,2X]|\leq \frac{C_2}{\phi(q^2)}\cdot\frac{X}{(\log X)^{\frac32}}
\end{equation}
for some constants $C_1,C_2>0$.

Suppose that $\cA$ is a subset of $\cP_2$ with a positive relatively upper density.
Let $$\delta_0=\frac{\overline{d}_{\cP_2}(\cA)}{2}.$$
Assume that $X$ is sufficiently large and
$$
\big|\cA\cap[X,2X]\big|\geq\delta_0\big|\cP_2\cap[X,2X]\big|.
$$
Let
$$
\eta_0=\sum_{q=1}^{\infty}\frac{1}{\phi(q^2)}.
$$
Since $\eta_0<+\infty$, there exists $Q_0>0$ such that
$$
\sum_{q>Q_0}^{\infty}\frac{1}{\phi(q^2)}<\frac{C_1}{2C_2}\cdot\frac{\delta_0\eta_0}{2},
$$
i.e.,
$$
\sum_{q>Q_0}\big|\cR_q\cap[X,2X]\big|\leq\frac{\delta_0|\cP_2\cap[X,2X]|}{2}\leq \frac{|\cA\cap[X,2X]|}{2}.
$$
By the pigeonhole principle, there exists $1\leq q_0\leq Q_0$ such that
$$
\big|\cR_{q_0}\cap\cA\cap[X,2X]\big|\geq\frac{|\cA\cap[X,2X]|}{2Q_0}\geq\frac{\delta_0\eta_0}{2Q_0}\cdot\frac{C_1X}{(\log X)^{\frac32}}.
$$

Let $w:=w(X)$ be an increasing function which very slowly tends to $+\infty$ as $X\to+\infty$. and let
$$W=\prod_{p\leq w}p.$$ 
Let
$$
\cS_W=\{1\leq b\leq W:\,(q_0^2b+1,W)=1, (b,W)\text{ has no prime factor of the form }4k+3\}.
$$
By the Chinese remainder theorem, we have
$$
\cS_W=W\prod_{\substack{p\equiv 3\pmod{4}\\ p\mid W,\ p\nmid q_0}}\bigg(1-\frac{2}{p}\bigg)\prod_{\substack{p\equiv 3\pmod{4}\\ p\mid q_0}}\bigg(1-\frac{1}{p}\bigg)
\prod_{\substack{p\not\equiv 3\pmod{4}\\ p\mid W,\ p\nmid q_0}}\bigg(1-\frac{1}{p}\bigg)=O\bigg(\frac{q_0^{\frac12}\phi(W)^{\frac32}}{\phi(q_0)^{\frac12}W^{\frac12}}\bigg).
$$
Conversely, if a prime $p\equiv q_0^2b+1\pmod{q_0^2W}$ can be written as $p=q_0^2(x^2+y^2)+1$ with $(x,y)=1$, then clearly we must have $(b,W)$ has no prime factor of the form $4k+3$.
Hence 
$$
\sum_{b\in S_{W}}|\{p\in\cR_{q_0}\cap\cA\cap[X,2X]:\,q_0^{-2}(p-1)\equiv b\pmod{W}\}|\geq \frac{\delta_0\eta_0}{2Q_0}\cdot\frac{C_1X}{(\log X)^{\frac32}}.
$$
By the pigeonhole principle, there exists $b\in S_W$ such that
$$
|\{p\in\cR_{q_0}\cap\cA\cap[X,2X]:\,q_0^{-2}(p-1)\equiv b\pmod{W}\}|\geq \frac{\phi(q_0)^{\frac12}W^{\frac12}}{q_0^{\frac12}\phi(W)^{\frac32}}\cdot\frac{\delta_0\eta_0}{Q_0}\cdot\frac{C_3X}{(\log X)^{\frac32}}
$$
for some constant $C_3>0$. Note that $\delta_0,\eta_0,q_0,Q_0$ are all positive constants only depending on the subset $\cA$. Let
$$
C_0=\frac{\phi(q_0)^{\frac12}}{q_0^{\frac12}}\cdot\frac{\delta_0\eta_0}{Q_0}\cdot\frac{C_3}{2}.
$$

Let
$$
\epsilon_m=\frac{1}{4^m(m+4)!}.
$$
Let $N$ be a prime lying in $[q_0^{-2}W^{-1}\epsilon_m^{-1}X,q_0^{-2}W^{-1}\epsilon_m^{-1}X+X(\log X)^{-2}]$. According to the prime number theorem with a remainder term, such prime $N$ always exists.
As we have shown,
\begin{equation}\label{nq02Wnb1Rq0}
\big|\{n\in[\epsilon_mN,2\epsilon_mN]:\,q_0^2(Wn+b)+1\in\cR_{q_0}\cap\cA\}\big|\geq \frac{W^{\frac12}}{\phi(W)^{\frac32}}\cdot\frac{C_0X}{(\log X)^{\frac32}}.
\end{equation}
Define the Mobious-type function
$$
\mu_3(n):=
\begin{cases}
1,&\text{if } n=1,\\
(-1)^r,& \text{if }n=p_1 \dots p_r,\ p_1,\ldots,p_r\text{ are distinct primes with }p_i\equiv3\pmod{4},\\
0,& \text{otherwise.}
\end{cases}
$$
Let
$$
R=N^{\frac{1}{2^{m+4}m}}.
$$
Let $\chi: \R \to \R$ be a smooth function such that $\chi(0)=1$
and $\chi$ is supported on the interval $[-1, 1]$.
Define
$$
\Lambda_R(n):= \sum_{\substack{d\mid n}} \mu(d) \chi\bigg(\frac{\log d}{\log R}\bigg)
$$
and
$$
\Lambda_{R}^*(n):= \sum_{\substack{d\mid n}} \mu_3(d) \chi\bigg(\frac{\log d}{\log R}\bigg).
$$
Evidently, if $q_0^2(Wn+b)+1\in\cR_{q_0}$, then
$$
\Lambda_R\big(q_0^2(Wn+b)+1\big)\Lambda_R^*(Wn+b)=1.
$$

We are ready to define our pseudorandom measure. Let
$$
\alpha_0=\lim_{s\to 1}\frac1{s-1}\prod_{p\equiv3\pmod{4}}\bigg(1-\frac{1}{p^s}\bigg)^2
$$
By our knowledge on the Riemann $\zeta$-function and the Dirichlet $L$-function, we have $\alpha_0>0$ and
$$
\lim_{s\to 1}\frac1{s-1}\prod_{p\not\equiv3\pmod{4}}\bigg(1-\frac{1}{p^s}\bigg)^2=\frac1{\alpha_0}.
$$ 
For any $n\in[\epsilon_mN,2\epsilon_mN]$, let
$$
\nu(n)=
\frac{(\log R)^{\frac32}}{\alpha_0^{\frac12}\cdot C_{\chi}}\prod_{\substack{p\not\equiv 3\pmod{4}\\ p\mid W}}\bigg(1-\frac{1}{p}\bigg)\prod_{\substack{p\equiv 3\pmod{4}\\ p\mid W}}\bigg(1-\frac{1}{p}\bigg)^2\cdot\Lambda_R\big(q_0^2(Wn+b)+1\big)^2\Lambda_{R}^*(Wn+b)^2,
$$
and let $\nu(n)=1$ for the other $n\in\Z_N$, where $C_{\chi}>0$ is a constant only depending on $\chi$ which we shall see later. Let
$$
f(n)=\frac{(\log R)^{\frac32}}{\alpha_0^{\frac12}\cdot C_{\chi}}\prod_{\substack{p\not\equiv 3\pmod{4}\\ p\mid W}}\bigg(1-\frac{1}{p}\bigg)\prod_{\substack{p\equiv 3\pmod{4}\\ p\mid W}}\bigg(1-\frac{1}{p}\bigg)^2
$$
provided that $n\in[\epsilon_mN,2\epsilon_mN]$ and $q_0^2(Wn+b)+1\in\cR_{q_0}\cap\cA$, 
Also, set $f(n)=0$ for the other $n\in\Z_N$.
In view of (\ref{nq02Wnb1Rq0}), clearly $f(x)\leq\nu(x)$ for each $x\in\Z_N$ and
$$
\E\big(f(x)\big|x\in\Z_N\big)\geq\frac{\epsilon_m\upsilon_0C_0}{2^{2m+8}m^2\cdot\alpha_0^{\frac12}C_\chi} >0.
$$
where
$$
\upsilon_0=\min_{x\geq 2}\prod_{\substack{p\not\equiv 3\pmod{4}\\ p\leq x}}\bigg(1-\frac{1}{p}\bigg)^{-\frac12}\prod_{\substack{p\equiv 3\pmod{4}\\ p\leq x}}\bigg(1-\frac{1}{p}\bigg)^{\frac12}.
$$
Hence by Lemma \ref{CFZThm}, if $\nu$ obeys the $m$-linear forms condition, then
$$
\E\big(f(x)f(x+y)\cdots f\big(x+(m-1)y)\big|x,y\in\Z_N\big)\geq c_{m,\cA}
$$
for some constant $c_{m,\cA}>0$ only depending  on $m$ and $\cA$. Since $\epsilon_m<1/m$, $\cR_{q_0}\cap\cA$ contains an non-trivial arithmetic progression of length $m$.

According to Green and Tao's discussions in the proof of \cite[Proposition 9.8]{GT08}, in order to show $\nu$ obeys the $m$-linear forms condition, it suffices to prove the following Goldston-Y{\i}ld{\i}r{\i}m-type estimation.
\begin{proposition}\label{GY}  Suppose that $m$ and $h$ are two positive integers.  Let
$$\psi_i(\mathbf{x}) := \sum_{j=1}^h L_{ij} x_j + v_i,\qquad i=1,\ldots,m,$$ 
be linear forms over $\Z_N$ such that

\medskip\noindent
(i) the coefficients $L_{ij}$ are integers and $|L_{ij}| \leq \sqrt{w}/2$ for any $1\leq i\leq m$ and $1\leq j\leq h$;

\medskip\noindent
(ii) the $h$-tuples $(L_{ij})_{j=1}^h$ are never identically zero, and that no two $h$-tuples are rational multiples of each other.

\medskip\noindent
Write $$\theta_i:=q_0^2(W\psi_i+b)+1,\qquad\theta_{m+i}:=W\psi_i+b$$ for each $1\leq i\leq m$. Suppose that $\bB$ is a product $\prod_{i=1}^h I_i \subset \R^h$, where each $I_i\subseteq\R$ is an interval of length at least $R^{10m}$. Then

\begin{align}\label{ELambdaRTheta}
&\E\big(\Lambda_R(\theta_1(\mathbf{x}))^2 \dots \Lambda_R(\theta_m(\mathbf{x}))^2\Lambda_{R,w}^*(\theta_{m+1}(\mathbf{x}))^2\dots \Lambda_{R,w}^*(\theta_{2m}(\mathbf{x}))^2\big| \bx \in \bB\big)\notag\\
 =&\bigg(\big(C_{\chi}+o_w(1)\big)\cdot\frac{\alpha_0^{\frac12}}{(\log R)^{\frac32}}\prod_{\substack{p\not\equiv 3\pmod{4}\\ p\mid W}}\frac{p}{p-1}\cdot\prod_{\substack{p\equiv 3\pmod{4}\\ p\mid W}}\frac{p^{2}}{(p-1)^{2}}\bigg)^m,
\end{align}
where $C_{\chi}>0$ is a constant only depending on $\chi$, and $o_w(1)$ means a term which tends to $0$ as $w\to+\infty$.
\end{proposition}

\section{The Goldston-Y{\i}ld{\i}r{\i}m-type estimation}
\setcounter{lemma}{0}
\setcounter{theorem}{0}
\setcounter{corollary}{0}
\setcounter{remark}{0}
\setcounter{equation}{0}
\setcounter{conjecture}{0}

In this section, the proof of Proposition \ref{GY} will be given. We shall follow the way of Tao in \cite{Tao}.
\begin{lemma}\label{cpjpsjL}
 Let $c_1,c_2,\ldots,c_k$ be some integers. For each prime $p$, arrange the bounded complex numbers $c_{p,1},\ldots,c_{p,k}$ such that
\begin{equation}\label{cpjcjp1}
c_{p,j}=c_j+O(p^{-1})
\end{equation}
unless $p$ divides $W$. Then
\begin{equation}\label{cpjpsjG1}
\prod_{p \equiv 3 \pmod{4}}\bigg(1 -\sum_{j=1}^k \frac{c_{p,j}}{p^{s_j}}\bigg) =G_1\cdot\big(1+O(H)\big)\prod_{j=1}^k\big(\alpha_0(s_j-1)\big)^{\frac{1}{2}c_{j}},
\end{equation}
where $s_1,\ldots,s_k$ are bounded complex numbers with $\Re(s_j)\geq 1$ and
$$
G_1=\prod_{p \equiv 3 \pmod{4}}\frac{1-p^{-1}(c_{p,1}+\cdots+c_{p,k})}{(1-p^{-1})^{c_1+\cdots+c_k}},\quad
H=\log w\cdot\max_{1\leq j\leq k}\{|s_j-1|\}.
$$
Also, we have
\begin{equation}\label{cpjpsjG2}
\prod_{p \not\equiv 3 \pmod{4}}\bigg(1 -\sum_{j=1}^k \frac{c_{p,j}}{p^{s_j}}\bigg) =G_2\cdot\big(1+O(H)\big)\prod_{j=1}^k\big(\alpha_0^{-1}(s_j-1)\big)^{\frac{1}{2}c_{j}},
\end{equation}
where
$$
G_2=\prod_{p\not\equiv 3 \pmod{4}}\frac{1-p^{-1}(c_{p,1}+\cdots+c_{p,k})}{(1-p^{-1})^{c_1+\cdots+c_k}}.
$$
Furthermore, the implied constants in (\ref{cpjpsjG1}) and (\ref{cpjpsjG2}) only depend on $k$ and the bounds of those $c_{p,j}$ and $s_j$.
\end{lemma}
\begin{proof}
According to the definition of $\alpha_0$, we have
$$\prod_{p\equiv 3\pmod{4}}\prod_{j=1}^{k}\bigg(1-\frac1{p^{s_j}}\bigg)^{c_j}=\big(1+o_k(1)\big)\prod_{j=1}^{k}\big(\alpha_0(s_j-1)\big)^{\frac{1}{2}c_{j}}$$
and
$$\prod_{p\not\equiv 3\pmod{4}}\prod_{j=1}^{k}\bigg(1-\frac1{p^{s_j}}\bigg)^{c_j}=\big(1+o_k(1)\big)\prod_{j=1}^{k}\big(\alpha_0^{-1}(s_j-1)\big)^{\frac{1}{2}c_{j}}$$
So it suffices to show that
\begin{equation}\label{p341psjcpjG1}
\prod_{p\equiv 3\pmod{4}}\frac{1-(p^{-s_1}c_{p,1}+\cdots+
p^{-s_k}c_{p,k})}{(1-p^{-s_1})^{c_1}\cdots(1-p^{-s_k})^{c_k}}=G_1\cdot\big(1+O(H)\big)
\end{equation}
and
\begin{equation}\label{pn341psjcpjG2}
\prod_{p\not\equiv 3\pmod{4}}\frac{1-(p^{-s_1}c_{p,1}+\cdots+
p^{-s_k}c_{p,k})}{(1-p^{-s_1})^{c_1}\cdots(1-p^{-s_k})^{c_k}}=G_2\cdot\big(1+O(H)\big).
\end{equation}
Here we only prove (\ref{p341psjcpjG1}), since (\ref{pn341psjcpjG2}) is very similar.
Suppose that $p\nmid W$. By (\ref{cpjcjp1}),
$$
\sum_{j=1}^k\frac{c_{p,j}}{p^{s_j}}=
\sum_{j=1}^k\frac{c_{j}}{p^{s_j}}+O\bigg(\frac1{p^2}\bigg)
$$
since $\Re(s_1),\ldots,\Re(s_k)\geq 1$.
Then we have
\begin{align*}
\frac{1-(p^{-s_1}c_{p,1}+\cdots+
p^{-s_k}c_{p,k})}{(1-p^{-s_1})^{c_1}\cdots(1-p^{-s_k})^{c_k}}=\frac{1-(p^{-s_1}c_{1}+\cdots+
p^{-s_k}c_{k})}{(1-p^{-s_1})^{c_1}\cdots(1-p^{-s_k})^{c_k}}+O_k(p^{-2}).
\end{align*}
Clearly for each $1\leq j\leq k$,
\begin{align*}
&\lim_{s_1,\ldots,s_k\to 1}\frac{\partial}{\partial s_j}\bigg(\frac{1-(p^{-s_1}c_{1}+\cdots+
p^{-s_k}c_{k})}{(1-p^{-s_1})^{c_1}\cdots(1-p^{-s_k})^{c_k}}\bigg)\\
=&
\lim_{s_1,\ldots,s_k\to 1}\log p\cdot\frac{p^{-s_j}c_j\cdot(1-p^{-s_j})-\big(1-(p^{-s_1}c_{1}+\cdots+
p^{-s_k}c_{k})\big)\cdot p^{-s_j}c_j}{(1-p^{-s_1})^{c_1}\cdots(1-p^{-s_j})^{c_j+1}\cdots(1-p^{-s_k})^{c_k}}\\
=&O\bigg(\frac{\log p}{p^2}\bigg).
\end{align*}
It follows that
$$\frac{1-(p^{-s_1}c_{1}+\cdots+
p^{-s_k}c_{k})}{(1-p^{-s_1})^{c_1}\cdots(1-p^{-s_k})^{c_k}}
=\frac{1-p^{-1}(c_{1}+\cdots+
c_{k})}{(1-p^{-1})^{c_1+\cdots+c_k}}+O\bigg(\frac{\log p}{p^2}\cdot\max_{1\leq j\leq k}\{|s_j-1|\}\bigg).$$
as $s_1,\ldots,s_k$ tend to $1$.
Similarly, if $p$ divides $W$, we also have
\begin{align*}
&\frac{1-(p^{-s_1}c_{p,1}+\cdots+
p^{-s_k}c_{p,k})}{(1-p^{-s_1})^{c_1}\cdots(1-p^{-s_k})^{c_k}}\\
=&\frac{1-(p^{-s_1}c_{p,1}+\cdots+
p^{-s_k}c_{p,k})}{(1-p^{-s_1})^{c_{p,1}}\cdots(1-p^{-s_k})^{c_{p,k}}}\cdot\prod_{j=1}^k(1-p^{-s_j})^{c_{p,j}-c_j}
\\
=&\frac{1-(p^{-1}c_{p,1}+\cdots+
p^{-s_k}c_{p,k})}{(1-p^{-1})^{c_1+\cdots+c_k}}+O\bigg(\max_{1\leq j\leq k}\{|s_j-1|\}\cdot\frac{\log p}{p}\sum_{j=1}^k|c_{p,j}-c_j|\bigg)
\end{align*}
as $s_1,\ldots,s_k\to 1$.
Note that by the prime number theorem,
$$
\sum_{p\mid W}\frac{\log p}{p}=\sum_{p\leq w}\frac{\log p}{p}=O\big(\log w\big).
$$
Multiplying all these above estimates together, we may get (\ref{p341psjcpjG1}). 
\end{proof}
Now we are ready to prove Proposition \ref{GY}. 
Let
$$u_j(n):=
\begin{cases}
\mu(n),&\text{if }1\leq j\leq m,\\
\mu_3(n),&\text{if }m+1\leq j\leq 2m.
\end{cases}.
$$
Clearly the left side of (\ref{ELambdaRTheta}) coincides with
$$\E\bigg(\prod_{j=1}^{2m} \sum_{\substack{d_j,e_j \leq R \\ d_j,e_j | \theta_j(\mathbf{x})}}u_j(d_j)u_j(e_j)\chi\bigg(\frac{\log d_j}{\log R}\bigg)\chi\bigg(\frac{\log e_j}{\log R}\bigg)\bigg|\mathbf{x}\in\bB\bigg),$$ 
which can be rearranged as
$$\sum_{\substack{d_j,e_j \leq R\\
\mu(d_j)\mu(e_j)\neq 0}}\prod_{j=1}^{2m}u_j(d_j)u_j(e_j)\chi\bigg(\frac{\log d_j}{\log R}\bigg)\chi\bigg(\frac{\log e_j}{\log R}\bigg)\cdot\E\bigg(\prod_{j=1}^{2m}\mathbf{1}_{d_j,e_j|\theta_j(\mathbf{x})}|\mathbf{x}\in\bB \bigg)$$
Assume that $d_1,\ldots,d_{2m},e_1,\ldots,e_{2m}$ are all square-free integers lying in $[1,R]$. Let
$D=[d_1,\ldots,d_{2m},e_{1},\ldots,e_{2m}]$ be the least common multiple of those $d_j,e_j$. 
Let $$\cI_{d_1,\ldots,d_{2m}}(p):= \{ 1 \leq j \leq 2m: p | d_j \}$$ 
and $$
\lambda_\cI^{*}(p) :=p\cdot\E\bigg( \prod_{j \in\cI} {\mathbf{1}}_{p\mid\theta_j(\bx)}\bigg|\bx \in \Z_p^h \bigg).$$
Suppose that $p\mid D$.
First, assume that $p\nmid W$. For $1\leq i<j\leq 2m$, since $\theta_i$ is not a rational multiple of $\theta_j$, we must have
$$
|\{\bx\in\Z_p^h:\,p\text{ divides both }\theta_i(\bx)\text{ and }\theta_j(\bx)\}|=O(p^{h-2}).
$$
So $\lambda_\cI^*(p)=O(p^{-1})$ whenever $|\cI|\geq 2$. Of course, if $|\cI|=1$, it is easy to see that $\lambda_\cI^*(p)=1+O(p^{-1})$.
Next, assume that $p\nmid W$. Recall that $(q_0^2b+1,W)=1$ and $(b,W)$ has no prime factor of the form $4k+3$.
Then $\lambda_\cI^*(p)=0$ provided that $p\equiv 3\pmod{4}$, or $p\equiv 1\pmod{4}$ and $\cI\cap\{1,\ldots,m\}\neq\emptyset$. Of course, if $p\equiv 1\pmod{4}$ and $\cI\subseteq\{m+1,\ldots,2m\}$, we still have $\lambda_\cI^*(p)=1+O(p^{-1})$ or $O(p^{-1})$ according to whether $|\cI|=1$ or $|\cI|\geq 2$.

Clearly $D\leq R^{4m}$. Note that $\mathbf{1}_{d_j,e_j|\theta_j(\mathbf{x})}$ can  be viewed as a function over $\Z_D^h$. So by the Chinese remainder theorem, we have
\begin{align*}
\E\bigg(\prod_{j=1}^{2m}\mathbf{1}_{d_j,e_j|\theta_j(\mathbf{x})}|\mathbf{x}\in\bB \bigg)=& \E\bigg(\prod_{j=1}^{2m}\mathbf{1}_{d_j,d_j'|\theta_j(\mathbf{x})}|\mathbf{x} \in\Z_D^h \bigg)+O_{m,h}\bigg(\frac{D}{\min_{1\leq i\leq h}|I_i|}\bigg)\\
=&\prod_{p \text{ prime}} \frac {\lambda_{\cI_{d_1,\ldots,d_{2m}}(p) \cup\cI_{e_1,\ldots,e_{2m}}(p)}^{*}(p)}{p}+O_{m,h}(R^{-6m}).
\end{align*}
It follows that
\begin{align*}
&\sum_{\substack{d_j,e_j \leq R\\
\mu(d_j)\mu(e_j)\neq 0}}\prod_{j=1}^{2m}u_j(d_j)u_j(e_j)\chi\bigg(\frac{\log d_j}{\log R}\bigg)\chi\bigg(\frac{\log e_j}{\log R}\bigg)\cdot\E\bigg(\prod_{j=1}^{2m}\mathbf{1}_{d_j,e_j|\theta_j(\mathbf{x})}|\mathbf{x}\in\bB \bigg)\\
=&\sum_{\substack{d_j,e_j \leq R\\
\mu(d_j)\mu(e_j)\neq 0}}\prod_{j=1}^{2m}u_j(d_j)u_j(e_j)\chi\bigg(\frac{\log d_j}{\log R}\bigg)\chi\bigg(\frac{\log e_j}{\log R}\bigg)\prod_p \frac {\lambda_{\cI_{d_1,\ldots,d_{2m}}(p) \cup \cI_{e_1,\ldots,e_{2m}}(p)}^{*}(p)}{p}+O_{m,h}(R^{-2m})\\
=&\sum_{\substack{d_j,e_j \leq R\\
\mu(d_j)\mu(e_j)\neq 0}}\prod_{j=1}^{2m}u_j(d_j)u_j(e_j)\chi\bigg(\frac{\log d_j}{\log R}\bigg)\chi\bigg(\frac{\log e_j}{\log R}\bigg)\cdot\frac{g(d_1, \ldots ,d_{2m}, e_1, \ldots ,e_{2m})}{[d_1,\ldots,d_{2m},e_1,\ldots,e_{2m}]},
\end{align*}
where $$g(d_1,\ldots,d_{2m},e_1,\ldots,e_{2m}):=\prod_{p\text{ prime}} \lambda_{\cI_{d_1,\ldots,d_{2m}}(p) \cup \cI_{e_1,\ldots,e_{2m}}(p)}(p)$$
and
$$\lambda_{\cI_{d_1,\ldots,d_{2m}}(p) \cup \cI_{e_1,\ldots,e_{2m}}(p)}(p):=
\begin{cases}
\lambda^{*}_{\cI_{d_1,\ldots,d_{2m}}(p) \cup \cI_{e_1,\dots,e_{2m}}(p)}(p),&\text{if } p|D,\\
1,&\text{if }p\nmid D.
\end{cases}
$$
As we have shown,
$$\lambda_{\cI}(p)=
\begin{cases}
1,&\text{if }\cI=\emptyset,\\
0,&\text{if }p\mid W\text{ and }p\equiv 3\pmod{4},\\
0,&\text{if }p\mid W, p\equiv 1\pmod{4}\text{ and }\cI\cap\{1,\ldots,m\}\neq\emptyset,\\
\lambda_{\cI}+O(p^{-1}),&\text{otherwise},
\end{cases}$$
where 
$$\lambda_\cI:=\begin{cases}
1,&\text{if }|\cI|=1,\\
0,&\text{if }|\cI|\geq2.
\end{cases}
$$

Write
$$e^x\chi(x)=\int_{-\infty}^{+\infty}\psi(t)e^{-ixt}dt$$
for some function $\psi$.
We know that $\psi$ is rapidly decreasing, i.e., obeys the bounds $$\psi(t)=O_A\big((1+|t|)^{-A}\big)$$ for any $A>0$. Then
$$\chi\bigg(\frac{\log d_j}{\log R}\bigg)=\int_{-\infty}^{+\infty}d_j^{-\frac{1+it}{\log R}}\psi(t)dt=\int_{(\log R)^{\frac12}}^{-(\log R)^{\frac12}}d_j^{-\frac{1+it}{\log R}}\psi(t)dt+O_{A,\chi}\big((\log R)^{-A}\big)$$
for any arbitrarily large $A$.
It is easy to see that
$$
\sum_{\substack{d_j,e_j \leq R\\
\mu(d_j)\mu(e_j)\neq 0}}\frac{1}{[d_1,\ldots,d_{2m},e_1,\ldots,e_{2m}]}
\leq\prod_{\substack{p\text{ prime}\\ p\leq R}}\bigg(1+\frac{4m}{p}\bigg)=O_m\big((\log R)^{4m}\big).
$$
Also, note that $g(d_1,\ldots,d_{2m},e_1,\ldots,e_{2m})$ is bounded. Hence for any large $A>0$,
\begin{align*}
&\sum_{\substack{d_j,e_j \leq R\\
\mu(d_j)\mu(e_j)\neq 0}}\frac{g(d_1, \ldots ,d_{2m}, e_1, \ldots ,e_{2m})}{[d_1,\ldots,d_{2m},e_1,\ldots,e_{2m}]}\prod_{j=1}^{2m}u_j(d_j)u_j(e_j)\chi\bigg(\frac{\log d_j}{\log R}\bigg)\chi\bigg(\frac{\log e_j}{\log R}\bigg)\\
=&\int_{-(\log R)^{\frac12}}^{(\log R)^{\frac12}}\cdots \int_{-(\log R)^{\frac12}}^{(\log R)^{\frac12}}\Omega(s_1,\ldots,s_{2m},t_1,\ldots,t_{2m})\prod_{j=1}^{2m}\psi(s_j)\psi(t_j)d s_jdt_j+O_{A,\chi}\big((\log R)^{-A}\big),
\end{align*}
where
$$
\Omega(s_1,\ldots,s_{2m},t_1,\ldots,t_{2m})=
\sum_{\substack{d_1,\ldots,d_{2m}\\ e_1,\ldots,e_{2m}}}\frac{g(d_1, \ldots ,d_{2m}, e_1, \ldots ,e_{2m})}{[d_1,\ldots,d_{2m},e_1,\ldots,e_{2m}]}\prod_{j=1}^{2m}\frac{u_j(d_j)u_j(e_j)}{d_j^{\frac{1+is_j}{\log R}}e_j^{\frac{1+it_j}{\log R}}}.
$$

Write  $y_j=(1+is_j)/\log R$ and $z_j=(1+it_j)/\log R$ for $1\leq j\leq 2m$. It is easy to see that
\begin{align*}
\Omega(s_1,\ldots,s_{2m},t_1,\ldots,t_{2m})
=&\sum_{\substack{d_1,\ldots,d_{2m}\\ e_1,\ldots,e_{2m}}}\frac{g(d_1, \ldots ,d_{2m}, e_1, \ldots ,e_{2m})}{[d_1,\ldots,d_{2m},e_1,\ldots,e_{2m}]}\prod_{j=1}^{2m}\frac{u_j(d_j)u_j(e_j)}{d_j^{y_j}e_j^{z_j}}\\
=&\prod_{\substack{p \equiv 3\pmod{4}}}\bigg(1+\sum_{\substack{\cI,\cJ\subseteq \{1,\dots,2m\}\\\cI\cup \cJ\neq \emptyset}}\frac{(-1)^{|\cI|+|\cJ|}\lambda_{\cI\cup \cJ}(p)}{p^{1+\sum_{j\in\cI}y_j+\sum_{j\in\cJ}z_j}}\bigg)
\\&\cdot\prod_{\substack{p \not\equiv 3\pmod{4}}}\bigg(1+\sum_{\substack{\cI,\cJ\subseteq \{1,\dots,m\}\\\cI\cup \cJ\neq \emptyset}}\frac{(-1)^{|\cI|+|\cJ|}\lambda_{\cI\cup \cJ}(p)}{p^{1+\sum_{j\in \cI}y_j+\sum_{j\in \cJ}z_j}}\bigg).
\end{align*}
Let
$$
\eta_1=\sum_{\substack{\cI,\cJ\subseteq\{1,\dots,2m\}\\\cI\cup\cJ\neq\emptyset}}(-1)^{|\cI|+|\cJ|}\lambda_{\cI\cup \cJ},\quad\eta_2=\sum_{\substack{\cI,\cJ\subseteq\{1,\dots,m\}\\\cI\cup \cJ\neq\emptyset}}(-1)^{|\cI|+|\cJ|}\lambda_{\cI\cup \cJ},
$$
and
$$
\kappa_1(p)=\sum_{\substack{\cI,\cJ\subseteq\{1,\dots,2m\}\\\cI\cup \cJ\neq\emptyset}}(-1)^{|\cI|+|\cJ|}\lambda_{\cI\cup \cJ}(p),\quad\kappa_2(p)=\sum_{\substack{\cI,\cJ\subseteq\{1,\dots,m\}\\\cI\cup \cJ\neq\emptyset}}(-1)^{|\cI|+|\cJ|}\lambda_{\cI\cup \cJ}(p).
$$
Applying Lemma \ref{cpjpsjL}, we get
\begin{align*}
&\prod_{\substack{p \equiv 3\pmod{4}}}\bigg(1-\sum_{\substack{\cI,\cJ\subseteq \{1,\dots,2m\}\\\cI\cup \cJ\neq \emptyset}}\frac{(-1)^{|\cI|+|\cJ|+1}\lambda_{\cI\cup \cJ}(p)}{p^{1+\sum_{j\in \cI}y_j+\sum_{j\in \cJ}z_j}}\bigg)\\
=&G_1\cdot\bigg(1+O\bigg(\frac{\log w}{(\log R)^{\frac12}}\bigg)\bigg)\prod_{\substack{\cI,\cJ\in \{1,\dots,2m\}\\\cI\cup \cJ\neq\emptyset}}\bigg(\alpha_0\cdot\bigg(\sum_{j\in \cI}y_j+\sum_{j\in \cJ}z_j\bigg)\bigg)^{(-1)^{|\cI|+|\cJ|+1}\cdot\frac{1}{2}\lambda_{\cI\cup\cJ}},
\end{align*}
where
$$G_1=\prod_{\substack{p \equiv 3\pmod{4}}}\bigg(1+\frac{\kappa_1(p)}{p}\bigg)\cdot\bigg(1-\frac{1}{p}\bigg)^{\eta_1}.
$$
Similarly,
\begin{align*}
&\prod_{\substack{p\not\equiv 3\pmod{4}}}\bigg(1-\sum_{\substack{\cI,\cJ\subseteq \{1,\dots,m\}\\ \cI\cup \cJ\neq \emptyset}}\frac{(-1)^{|\cI|+|\cJ|+1}\lambda_{\cI\cup \cJ}(p)}{p^{1+\sum_{j\in \cI}y_j+\sum_{j\in \cJ}z_j}}\bigg)\\
=&G_2\cdot\bigg(1+O\bigg(\frac{\log w}{(\log R)^{\frac12}}\bigg)\bigg)\prod_{\substack{\cI,\cJ\in \{1,\dots,m\}\\ \cI\cup \cJ\neq\emptyset}}\bigg(\alpha_0^{-1}\cdot\bigg(\sum_{j\in \cI}y_j+\sum_{j\in \cJ}z_j\bigg)\bigg)^{(-1)^{|\cI|+|\cJ|+1}\cdot\frac{1}{2}\lambda_{\cI\cup \cJ}},
\end{align*}
where
$$G_2=\prod_{\substack{p\not\equiv 3\pmod{4}}}\bigg(1+\frac{\kappa_2(p)}{p}\bigg)\cdot\bigg(1-\frac{1}{p}\bigg)^{\eta_2}.
$$

Suppose that $p\equiv 3\pmod{4}$. Clearly $\kappa_1(p)=0$ if $p\mid W$. Assume that $p\nmid W$. By (\ref{cpjcjp1}),
$$
\kappa_1(p)=\sum_{\substack{\cI,\cJ\subseteq\{1,\dots,2m\}\\ |\cI\cup \cJ|=1}}(-1)^{|\cI|+|\cJ|}+O_m(p^{-1})=-2m+O_m(p^{-1}).
$$
Similarly, we also have $\eta_1=-2m$. So
\begin{align*}
G_1=&\prod_{\substack{p \equiv 3\pmod{4}\\ p\mid W}}\bigg(1-\frac{1}{p}\bigg)^{-2m}\cdot\prod_{\substack{p \equiv 3\pmod{4}\\ p\nmid W}}\bigg(1-\frac{2m+O(p^{-1})}{p}\bigg)\bigg(1-\frac{1}{p}\bigg)^{-2m}\\
=&\big(1+o_{w}(1)\big)\prod_{\substack{p \equiv 3\pmod{4}\\ p\mid W}}\bigg(1-\frac{1}{p}\bigg)^{-2m}.
\end{align*}

Suppose that $p\not\equiv 3\pmod{4}$. Clearly we still have $\kappa_2(p)=0$ for those $p\mid W$. If $p\nmid W$, then
$$
\kappa_2(p)=\sum_{\substack{\cI,\cJ\subseteq\{1,\dots,m\}\\ |\cI\cup \cJ|=1}}(-1)^{|\cI|+|\cJ|}+O_m(p^{-1})=-m+O_m(p^{-1}).
$$
Also, $\eta_2=-m$. Thus
\begin{align*}
G_2=\big(1+o_{w}(1)\big)\prod_{\substack{p \not\equiv 3\pmod{4}\\ p\mid W}}\bigg(1-\frac{1}{p}\bigg)^{-m}.
\end{align*}

On the other hand, clearly
\begin{align*}
&\prod_{\substack{\cI,\cJ\in \{1,\dots,2m\}\\ \cI\cup \cJ\neq\emptyset}}\bigg(\alpha_0\cdot\bigg(\sum_{j\in \cI}y_j+\sum_{j\in \cJ}z_j\bigg)\bigg)^{(-1)^{|\cI|+|\cJ|+1}\cdot\frac{1}{2}\lambda_{\cI\cup \cJ}}\\
=&\bigg(\frac{\alpha_0}{\log R}\bigg)^{-\eta_1}\prod_{\substack{\cI,\cJ\in \{1,\dots,2m\}\\ \cI\cup \cJ\neq\emptyset}}\bigg(\log R\sum_{j\in \cI}y_j+\log R\sum_{j\in \cJ}z_j\bigg)^{(-1)^{|\cI|+|\cJ|+1}\cdot\frac{1}{2}\lambda_{\cI\cup \cJ}}\\
=&\frac{\alpha_0^m}{(\log R)^m}\cdot\Psi_1(s_1,\ldots,s_{2m},t_1,\ldots,t_{2m}),
\end{align*}
where
$$
\Psi_1(s_1,\ldots,s_{2m},t_1,\ldots,t_{2m})=\prod_{\substack{\cI,\cJ\in \{1,\dots,2m\}\\ \cI\cup \cJ\neq\emptyset}}\bigg(\sum_{j\in \cI}(1+is_j)+\sum_{j\in \cJ}(1+it_j)\bigg)^{(-1)^{|\cI|+|\cJ|+1}\cdot\frac{1}{2}\lambda_{\cI\cup \cJ}}.
$$
Similarly, since $\eta_2=-m$, we also have
\begin{align*}
&\prod_{\substack{\cI,\cJ\in \{1,\dots,m\}\\ \cI\cup \cJ\neq\emptyset}}\bigg(\alpha_0\cdot\bigg(\sum_{j\in \cI}y_j+\sum_{j\in \cJ}z_j\bigg)\bigg)^{(-1)^{|\cI|+|\cJ|+1}\cdot\frac{1}{2}\lambda_{\cI\cup \cJ}}\\
=&\frac{1}{\alpha_0^{\frac12m}(\log R)^{\frac12m}}\cdot\Psi_2(s_1,\ldots,s_{m},t_1,\ldots,t_{m}),
\end{align*}
where
$$
\Psi_2(s_1,\ldots,s_{m},t_1,\ldots,t_{m})=\prod_{\substack{\cI,\cJ\in \{1,\dots,m\}\\\cI\cup \cJ\neq\emptyset}}\bigg(\sum_{j\in \cI}(1+is_j)+\sum_{j\in \cJ}(1+it_j)\bigg)^{(-1)^{|\cI|+|\cJ|+1}\cdot\frac{1}{2}\lambda_{\cI\cup \cJ}}.
$$
Thus we get
\begin{align*}
&\Omega(s_1,\ldots,s_{2m},t_1,\ldots,t_{2m})\\
=&\big(1+o_{w}(1)\big)\prod_{\substack{p \equiv 3\pmod{4}\\ p\mid W}}\bigg(1-\frac{1}{p}\bigg)^{-2m}\prod_{\substack{p\not\equiv 3\pmod{4}\\ p\mid W}}\bigg(1-\frac{1}{p}\bigg)^{-m}\\
&\cdot \frac{\alpha_0^{\frac12m}}{(\log R)^{\frac32m}}\cdot\Psi_1(s_1,\ldots,s_{m},t_1,\ldots,t_{m})\Psi_2(s_1,\ldots,s_{m},t_1,\ldots,t_{m}).
\end{align*}
Let
$$
C_{m,\chi}=\int_{-(\log R)^{\frac12}}^{(\log R)^{\frac12}}\cdots \int_{-(\log R)^{\frac12}}^{(\log R)^{\frac12}}\prod_{k=1}^2\Psi_k(s_1,\ldots,s_{m},t_1,\ldots,t_{m})\cdot\prod_{j=1}^{2m}\psi(s_j)\psi(t_j)ds_jdt_j.
$$
It is easy to see that 
$$
C_{m,\chi}=C_{\chi}^m+O_A\big((\log R)^{-A}\big)
$$
for any $A>0$, where
$$
C_{\chi}=\iint_{\R\times\R}(1+is)^{\frac32}(1+it)^{\frac32}(2+it+is)^{-\frac32}\psi(t)\psi(s)ds dt.
$$
Finally, recalling that $0\leq f(x)\leq\nu(x)$ and 
$\E(f(x)|x\in\Z_N)>0$, clearly we must have $C_{\chi}>0$.
Then Proposition \ref{GY} is concluded.
\qed

\begin{acknowledgment} We thank Professor Henryk Iwaniec for his helpful explanation on Theorem 1 of \cite{Iwaniec72}.
\end{acknowledgment}

     \end{document}